\documentclass[10pt]{amsart}
\usepackage{latexsym}
\usepackage{amsfonts,amsmath,amssymb,indentfirst}
\usepackage[pdftex]{hyperref}
\newcommand{\beq}{\begin{equation}}
\newcommand{\eeq}{\end{equation}}
\newcommand{\bdism}{\begin{displaymath}}
\newcommand{\edism}{\end{displaymath}}

\newcommand{\setQ}{\mathbb Q}
\newcommand{\setR}{\mathbb R}
\newcommand{\setC}{\mathbb C}

\newtheorem{theorem}{Theorem}[section]
\newtheorem{proposition}[theorem]{Proposition}

\newtheorem{lemma}[theorem]{Lemma}

\newtheorem{definition}{Definition}

\thanks{* Department of Mathematics, Duke University, Durham NC 27708-0320, USA \\
Phone: (919) 660-2850, Fax: (919) 660-2821, E-mail: luca@math.duke.edu}

\author[Di Cerbo]{Luca Fabrizio Di Cerbo}

\title{On K\"ahler-Einstein surfaces with edge singularities}
\begin{document}
\maketitle

\begin{abstract}
In this paper we characterize logarithmic surfaces which admit
K\"ahler-Einstein metrics with negative scalar curvature and small
edge singularities along a normal crossing divisor.
\end{abstract}

\tableofcontents

\section{Introduction}
\pagenumbering{arabic}

In a celebrated paper \cite{Yau}, S.-T. Yau solved the Calabi
conjecture by studying complex Monge-Amp\`ere equations on compact
K\"ahler manifolds.

The solution of the Calabi conjecture provides a general existence
theorem for \emph{K\"ahler-Einstein} metrics of negative or zero
scalar curvature (in the negative case, independently, T. Aubin
\cite{Aubin}).

Since then, complex Monge-Amp\`ere equations have been extensively
studied. In particular, major developments in the theory of complex
Monge-Amp\`ere equations with singular right hand side have
occurred, see for example \cite{Bedford}, \cite{Kolo}, \cite{Pali}.
As in the non-degenerate case, these analytical results provide very
general existence theorems for \emph{singular} K\"ahler-Einstein
metrics. For more details, results and the connection with the
theory of the minimal model see also \cite{Guedj} and the
bibliography therein.

Such existence theorems, although extremely general, provide little
control on the asymptotic behavior of the singular K\"ahler-Einstein
metric near the \emph{degeneracy} locus. This fact somehow limits
many of the geometric applications.

In \cite{Tian}, G. Tian proposed to look for K\"ahler-Einstein
metrics with \emph{cone} singularities along a normal crossing
divisor. Such metrics should still arise as solutions of certain
singular complex Monge-Amp\`ere equations, but because of the mild
type of singularities developed by the associated K\"ahler-Einstein
potentials one should be able to derive many interesting geometric
consequences, see again \cite{Tian}.

In light of the recent advances and interest in the theory of
K\"ahler-Einstein metrics with edge singularities \cite{Donaldson},
\cite{Campana1}, \cite{Mazzeo}, \cite{LeBrun}, we study the geometry
of logarithmic surfaces which admit K\"ahler-Einstein metrics of
negative scalar curvature and edge singularities of small cone
angles along a normal crossing divisor.

The paper is organized as follows. In Section \ref{demailly}, we use
the theory of K\"ahler currents to derive some generalities
regarding the existence of K\"ahler-Einstein metrics with edge
singularities. In Section \ref{sakai}, we show how the theory of
\emph{semi-stable} curves on algebraic surfaces developed by F.
Sakai in \cite{Sakai} is relevant for the study of K\"ahler-Einstein
metrics with edges on algebraic surfaces. Section \ref{twisted}
contains a proof of the geometric semi-positivity of certain twisted
log-canonical bundles associated to a logarithmic surface. Finally,
in Section \ref{applications}, we show how the results in
\cite{Campana1} and \cite{Mazzeo} can be used to classify
logarithmic surfaces which admit negative K\"ahler-Einstein metrics
with \emph{small} edge singularities. Moreover, we briefly discuss
the Chern-Weil approach, through K\"ahler-Einstein metrics with
cone-edge singularities, to the logarithmic
\emph{Bogomolov-Miyaoka-Yau} inequality for surfaces of log-general
type. Remarkably, this is intimately connected with the recent work
of M. F. Atiyah and C. LeBrun in \cite{LeBrun}.

\section{K\"ahler-Einstein metrics with edge singularities as K\"ahler
currents}\label{demailly}

Let $\overline{M}$ be a $n$-dimensional projective manifold and $D$
a normal crossing divisor. Given the pair $(\overline{M}, D)$, let
us define the notion of a K\"ahler metric with edge singularities
along $D$.

Let $\{D_{i}\}$ be the irreducible components of $D$ and
$\{\alpha_{i}\}$ a collection of positive numbers less than one and
bigger than zero. A smooth K\"ahler metric $\hat{\omega}$ on
$\overline{M}\backslash D$ is said to have edge singularities of
cone angles $2\pi(1-\alpha_{i})$ along the $D_{i}$'s if for any
point $p\in D$ there exists a coordinate neighborhood
$(\Omega;z_{1},...,z_{n})$ where
\begin{align}\notag
D_{|_{\Omega}}=z_{1}\cdot...\cdot z_{k}=0
\end{align}
and a positive constant $C$ such that
\begin{align}\label{Jeffres-Tian}
C^{-1}\omega_{0}\leq\hat{\omega}\leq C\omega_{0}
\end{align}
where
\begin{align}\label{model}
\omega_{0}=\sqrt{-1}\bigg(\sum^{k}_{i=1}\frac{dz_{i}\wedge
d\overline{z}_{i}}{|z_{i}|^{2\alpha_{i}}}+\sum^{n}_{i=k+1}dz_{i}\wedge
d\overline{z}_{i}\bigg).
\end{align}
Summarizing, from the K\"ahler geometry point of view on
$\overline{M}\backslash D$ the \emph{smooth} form $\hat{\omega}$ is
simply an \emph{incomplete} K\"ahler metric with finite volume.
Because of the finite volume property, standard results in the
theory of currents \cite{Demailly2} imply that $\hat{\omega}$ can be
regarded as a strictly positive closed real $(1, 1)$-current on
$\overline{M}$. In other words, $\hat{\omega}$ is a \emph{K\"ahler
current} on $\overline{M}$ with \emph{singular support} $D$.

We are now ready to introduce the definition of a K\"ahler-Einstein
metric with edge singularities.

\begin{definition}[Tian \cite{Tian}]
A K\"ahler current $\hat{\omega}$ with edges of cone angles
$2\pi(1-\alpha_{i})$ along the $D_{i}$'s is called K\"ahler-Einstein
with curvature $\lambda$ if it satisfies the distributional equation
\begin{align}\label{Tian}
Ric_{\hat{\omega}}-\sum_{i}\alpha_{i}[D_{i}]=\lambda\hat{\omega}
\end{align}
where by $[D]$ we indicate the current of integration along $D$.
\end{definition}

In what follows, we will focus on the negatively curved case. Thus,
let $\hat{\omega}$ be a K\"ahler-Einstein metrics with edge
singularities as in \ref{Tian} with $\lambda=-1$. By the
Poincar\'e-Lelong formula \cite{Griffiths}, the current
$-Ric_{\hat{\omega}}+\sum_{i}\alpha_{i}[D_{i}]$ represents the
cohomology class of the $\setR$-divisor
$K_{\overline{M}}+\sum_{i}\alpha_{i}D_{i}$. Since by assumption
$\hat{\omega}$ satisfies \ref{Tian} with $\lambda=-1$, we have that
the cohomology class of the $\setR$-divisor
$K_{\overline{M}}+\sum_{i}\alpha_{i}D_{i}$ can be represented by a
K\"ahler current. By the structure of the \emph{pseudo-effective}
cone given in \cite{Demailly1}, we conclude that
$K_{\overline{M}}+\sum_{i}\alpha_{i}D_{i}$ is a \emph{big}
$\setR$-divisor.

Next, we want to show that
$K_{\overline{M}}+\sum_{i}\alpha_{i}D_{i}$ has to be an \emph{ample}
$\setR$-divisor. To this aim, recall that given a closed positive
$(1, 1)$-current $T$ in $\setC^{n}$ one can define the Lelong
numbers \cite{Griffiths}. More precisely, let
\begin{align}\notag
\nu(T, r, x)=\frac{1}{(\pi r^{2})^{n-1}}\int_{B(x,
r)}T\wedge\omega^{n-1}
\end{align}
where $\omega$ is the standard Euclidean K\"ahler form. The Lelong
number of $T$ at $x$ is then simply defined as
\begin{align}\notag
\nu(T, x):=\lim_{r\rightarrow 0}\nu(T, r, x).
\end{align}

The theory of positive currents, as developed by Lelong, Siu and
others see for example \cite{Demailly2}, ensures that the definition
given above makes sense and that furthermore it can be extended to
currents on K\"ahler manifolds.

We are now interested in the computation of Lelong numbers of a
K\"ahler current with edge singularities. For simplicity we treat
the complex surface case only. The general case is completely
analogous.

Let $D^{*}$ denote the smooth part of the divisor $D$. Because of
the \emph{quasi-isometric} condition given in \ref{Jeffres-Tian},
for any $x\in D^{*}$ the computation of the Lelong numbers of
$\hat{\omega}$ reduces to the evaluation of
\begin{align}\notag
\lim_{r\rightarrow
0}\int_{B(r)}\frac{1}{|z_{1}|^{2\alpha}}dz_{1}\wedge
d\overline{z}_{1}\wedge dz_{2}\wedge d\overline{z}_{2}.
\end{align}
Since
\begin{align}\notag
\int\frac{dr}{r^{2\alpha-1}}=\frac{1}{2(1-\alpha)}r^{2(1-\alpha)}
\end{align}
and $0<\alpha<1$, we conclude that $\nu(\hat{\omega}, x)=0$ for any
$x\in D^{*}$. Furthermore, the Lelong numbers are clearly zero
outside the singular support of $\hat{\omega}$.

We are now ready to show that $\hat{\omega}$ represents an ample
class. The idea is to apply the celebrated \emph{regularization}
result of Demailly \cite{Demailly1}.

For cohomological reasons, on a compact K\"ahler manifold, the
Lelong numbers of a given positive current are bounded from above.
Let us then choose a positive number $c$ such that
$\nu(\hat{\omega}, x)<c$ for any point $x\in\overline{M}$. Next,
observe that there exists a smooth positive $(1, 1)$-form $\gamma$
such that $\hat{\omega}\geq\gamma$. Because of the choice of the
number $c$, the regularizations $\{\hat{\omega}_{c, k}\}$ given in
the main theorem of \cite{Demailly1} are smooth over $\overline{M}$
and satisfy
\begin{align}\notag
\hat{\omega}_{c, k}\geq \gamma-\mu_{c, k}\overline{\omega}
\end{align}
where $\{\mu_{c, k}\}$ is decreasing sequence of continuous
functions converging pointwise to $\nu(\hat{\omega}, x)$, and
$\overline{\omega}$ is a given background K\"ahler metric. As a
result, given any irreducible curve $C$ on $\overline{M}$, we have
\begin{align}\notag
(\hat{\omega}, C)=(\hat{\omega}_{c, k},
C)=\int_{C^{*}}\hat{\omega}_{c,
k}\geq\int_{C^{*}}\gamma-\int_{C^{*}}\mu_{c, k}\overline{\omega}
\end{align}
where by $C^{*}$ we denote the smooth part of $C$. Letting $k$ go to
infinity and applying the dominated convergence theorem of Lebesgue,
we conclude that $\hat{\omega}$ represents a \emph{strictly} nef
class. By a Nakai-Moishezon type criterion for $\setR$-divisors due
to F. Campana and T. Peternell \cite{Campana}, we conclude that
$K_{\overline{M}}+\sum_{i}\alpha_{i}D_{i}$ is indeed an ample
$\setR$-divisor.

Let us summarize these observations into a proposition.

\begin{proposition}\label{luca}
Let $\hat{\omega}$ be a K\"ahler-Einstein metric with edge
singularities as in \ref{Tian} with $\lambda=-1$, then the
$\setR$-divisor $K_{\overline{M}}+\sum_{i}\alpha_{i}D_{i}$ is ample.
\end{proposition}

An analogous proposition holds for K\"ahler-Einstein currents of
positive scalar curvature.

\begin{proposition}\label{fabrizio}
Let $\hat{\omega}$ be a K\"ahler-Einstein metric with edge
singularities as in \ref{Tian} with $\lambda=1$, then the
$\setR$-divisor $-(K_{\overline{M}}+\sum_{i}\alpha_{i}D_{i})$ is
ample.
\end{proposition}

For the zero scalar curvature case the cohomological restrictions
are obvious.

\section{On K\"ahler-Einstein surfaces with small edge singularities and negative scalar
curvature}\label{sakai}

Let $\overline{M}$ be a smooth projective surface. Let $D$ be a
reduced divisor having normal crossings on $\overline{M}$. The
\emph{logarithmic} canonical bundle associated to $D$ is then
defined as $\mathcal{L}=K_{\overline{M}}+D$, for details see
\cite{Iitaka}. Similarly, for any real number $\alpha\in(0,1]$ we
define $\mathcal{L}_{\alpha}=K_{\overline{M}}+\alpha D$. Let us
refer to the $\mathcal{L}_{\alpha}$'s as \emph{twisted}
log-canonical bundles.

In Section \ref{demailly}, we have shown that if a pair
$(\overline{M}, D)$ admits a K\"ahler-Einstein metric with edge
singularities of cone angle $2\pi(1-\alpha)$ along $D$ then the
$\setR$-divisor $\mathcal{L}_{\alpha}$ has to be ample, see
Proposition \ref{luca}. In this section we classify the pairs
$(\overline{M}, D)$ for which the corresponding
$\mathcal{L}_{\alpha}$'s are ample for all $\alpha$ close enough to
one. Furthermore, we give an analogous classification result under
the weaker requirement for $\mathcal{L}_{\alpha}$ to be big and nef
for all $\alpha$ close enough to one. This more general
classification result will be used in Section \ref{twisted} and in
the applications given in Section \ref{applications}. Let us start
with the following definitions, for details see \cite{Sakai}.

\begin{definition}
Let $\overline{M}$ be an algebraic surface and let $D$ be a reduced
normal crossing divisor. The pair $(\overline{M}, D)$ is said to be
$D$-minimal if there are no curves $E$ in $\overline{M}$ such that
\begin{align}\notag
E\simeq\setC P^{1}, \quad E^{2}=-1, \quad E\cdot D\leq 1.
\end{align}
\end{definition}

Observe that if $(\overline{M},D)$ is $D$-minimal then
$\overline{M}$ need not be minimal.

\begin{definition}
The divisor $D$ is called semi-stable if any  smooth rational
component in $D$ intersects the other components of $D$ at least in
two points.
\end{definition}

We can now prove the following lemma.

\begin{lemma}\label{lemma1}
If $\mathcal{L}_{\alpha}$ is big and nef for all values of $\alpha$
close enough to one, then the pair $(\overline{M},D)$ is $D$-minimal
of log-general type with $D$ a semi-stable curve.
\end{lemma}
\begin{proof}
Let us assume $\mathcal{L}_{\alpha}$ to be big and nef for all
values of $\alpha$ close to one. Let consider the limit
$\mathcal{L}_{\alpha}\rightarrow \mathcal{L}$ as $\alpha$ approaches
one. By Kleiman's theorem \cite{Lazarsfeld}, the closure of the
\emph{ample} cone is exactly the \emph{nef} cone. In particular, the
nef cone is closed and $\mathcal{L}$ is necessarily nef. Let $E$ be
an irreducible component of $D$, we then have
\begin{align}\notag
\mathcal{L}\cdot E=2p_{a}(E)-2+E(D-E)
\end{align}
and since by construction $\mathcal{L}$ is nef, we conclude that $D$
is \emph{semi-stable}. Moreover, the pair $(\overline{M}, D)$ is
$D$-minimal. Finally, it remains to show that $\mathcal{L}$ is a big
divisor. Since
\begin{align}\notag
\mathcal{L}=\mathcal{L}_{\alpha}+(1-\alpha)D
\end{align}
and by assumption $\mathcal{L}_{\alpha}$ is big, the result simply
follows from the structure of the algebraic pseudo-effective cone
\cite{Lazarsfeld}.
\end{proof}

Conversely, we can prove the following.

\begin{lemma}\label{lemma2}
Let $(\overline{M}, D)$ be a $D$-minimal semi-stable pair of
log-general type, then $\mathcal{L}_{\alpha}$ is big and nef for all
values of $\alpha$ close to one.
\end{lemma}
\begin{proof}

Since by assumption $\mathcal{L}$ is big and the set of big
$\setR$-divisors forms an open convex cone, $\mathcal{L}_{\alpha}$
is big for all values of $\alpha$ close enough to one, see Corollary
2.24. in \cite{Lazarsfeld}.

It remains to understand the \emph{nefness} properties of
$\mathcal{L}_{\alpha}$. Let $E$ be an irreducible divisor in
$\overline{M}$ and let us write
\begin{align}\label{dotto}
\mathcal{L}_{\alpha}\cdot E=\mathcal{L}\cdot E+(\alpha-1)D\cdot E.
\end{align}
Now, the semi-stability assumption on $D$ combined with the
$D$-minimality of $(\overline{M},D)$ implies that $\mathcal{L}$ is
nef. Thus, if $E$ is such that $\mathcal{L}\cdot E>0$, by making
$\alpha$ close enough to one we can always assume that
$\mathcal{L}_{\alpha}\cdot E>0$. The only curves we then need to
consider are the ones such that $\mathcal{L}\cdot E=0$. Thus, if $E$
is irreducible and such that $\mathcal{L}\cdot E=0$ we have the
following cases:
\begin{center}
\begin{itemize}
\item[-] $E\nsubseteq D$, $E\simeq\setC P^{1}$, $E^{2}=-2$;
\item[-] $E\subset D$, $E$ is an isolated component of $D$ with $p_{a}(E)=1$, or $E\simeq\setC P^{1}$ and
$E\cdot(D-E)=2$.
\end{itemize}
\end{center}
Let us study first the case when $E$ is an isolated component of
$D$. Observe that
\begin{align}\notag
\mathcal{L}\cdot E=2p_{a}(E)-2+D\cdot E-E^{2}
\end{align}
so that if $\mathcal{L}\cdot E=0$ and $p_{a}(E)=1$ then $D\cdot
E=E^{2}$. Since $\mathcal{L}^{2}>0$, by the Hodge index theorem we
conclude that $D\cdot E=E^{2}<0$. By using \ref{dotto}, we conclude
that $\mathcal{L}_{\alpha}$ dots positively with the isolated
components of $D$ with arithmetic genus equal to one.

Let now $E$ be a $(-2)$-curve not contained in the boundary divisor
$D$ such that $\mathcal{L}\cdot E=0$. Since in this case $D\cdot
E=0$ we have
\begin{align}\notag
\mathcal{L}_{\alpha}\cdot E=\mathcal{L}\cdot E+(\alpha-1)D\cdot E=0
\end{align}
constantly in $\alpha$.

Finally, let $E$ be a smooth rational component of $D$ such that
$E\cdot(D-E)=2$ and $\mathcal{L}\cdot E=0$. Again by the Hodge index
theorem we must have $E^{2}<0$. As a result
\begin{align}\notag
\mathcal{L}_{\alpha}\cdot E=(\alpha-1)D\cdot
E=(\alpha-1)(2+E^{2})\geq0
\end{align}
with equality iff $E$ is a $(-2)$-curve.
\end{proof}

Combining Lemma \ref{lemma1} and \ref{lemma2} we have then proved
the following theorem.

\begin{theorem}\label{luca1}
The $\setR$-divisor $\mathcal{L}_{\alpha}$ is big and nef for all
values of $\alpha$ close enough to one iff the pair
$(\overline{M},D)$ is $D$-minimal of log-general type with $D$ a
semi-stable curve.
\end{theorem}

Finally, we characterize when $\mathcal{L}_{\alpha}$ is ample for
all values of $\alpha$ close to one.

\begin{theorem}\label{luca2}
The $\setR$-divisor $\mathcal{L}_{\alpha}$ is ample for all values
of $\alpha$ close to one iff the pair $(\overline{M},D)$ is
$D$-minimal of log-general type without interior rational
$(-2)$-curves and $D$ is a semi-stable curve without $(-2)$-rational
curves intersecting the other components of $D$ in just two points.
\end{theorem}

\begin{proof}
Let us assume $\mathcal{L}_{\alpha}$ to be ample for all values of
$\alpha$ close to one. By Lemma \ref{lemma1} the pair
$(\overline{M},D)$ is minimal of log-general type and $D$
semi-stable. Finally, it is clear that the pair $(\overline{M},D)$
cannot contain rational $(-2)$-curves as above.

Conversely, if $(\overline{M},D)$ is as in statement, the ampleness
of $\mathcal{L}_{\alpha}$ follows from the computations in Lemma
\ref{lemma2} and the characterization of the amplitude for
$\setR$-divisors given in \cite{Campana}.
\end{proof}

\section{Geometric semi-positivity of twisted log-canonical
bundles}\label{twisted}

In Section \ref{sakai} we have shown that given a minimal
semi-stable log-general pair $(\overline{M}, D)$ there are
obstructions for the $\setR$-divisors $\mathcal{L}_{\alpha}$ to be
ample for all $\alpha$ close enough to one. Nevertheless one expects
good positivity properties for these $\setR$-divisors.

Recall the following definition, see for example \cite{Fujita}.

\begin{definition}
A line bundle $L$ is called geometrically semi-positive if
$c_{1}(L)$ can be represented by a smooth closed Hermitian $(1,
1)$-form which is everywhere positive semi-definite.
\end{definition}

In this section we want to show that, given a minimal semi-stable
log-general pair $(\overline{M}, D)$, the associated $\setR$-divisor
$\mathcal{L}_{\alpha}$ is geometrically semi-positive for all values
of $\alpha$ close enough to one. Moreover, we will precisely
describe the locus where these $\setR$-divisors fail to be ample.

The main tool used here will be a well-known theorem of Reider
\cite{Reider}. In fact, we apply this theorem to certain integer
multiplies of $\mathcal{L}_{\alpha}$ where the parameter $\alpha$ is
appropriately chosen to be rational and close enough to one.

Thus, let us start with $\setQ$-divisors of the form
\begin{align}\notag
\mathcal{L}_{\alpha_{n}}=K_{\overline{M}}+\frac{(n-2)}{n} D.
\end{align}
By clearing denominators we obtain
\begin{align}\label{adjoint}
n\mathcal{L}_{\alpha_{n}}=2K_{\overline{M}}+(n-2)\mathcal{L}=K_{\overline{M}}+(K_{\overline{M}}+(n-2)\mathcal{L}).
\end{align}

By construction the pair $(\overline{M}, D)$ is $D$-minimal of
log-general type and $D$ a semi-stable pair. The log-canonical
bundle $\mathcal{L}$ is then big and nef. Furthermore, for $n$ big
enough by Theorem \ref{luca1} the divisor
$\overline{\mathcal{L}}_{n}=K_{\overline{M}}+(n-2)\mathcal{L}$ is
big and nef.

Now, given a big and nef divisor $L$ on $\overline{M}$ the theorem
of Reider \cite{Reider} provides a powerful tool for the study of
the linear system $|K_{\overline{M}}+L|$. The idea is now to apply
this theorem to the linear system associated to the divisor given in
\ref{adjoint}.

By letting $n$ be big enough, we can always assume that
$\overline{\mathcal{L}}_{n}^{2}>4$. By the theorem of Reider, we
know that if $x\in\overline{M}$ is a base point of
$|K_{\overline{M}}+\overline{\mathcal{L}}_{n}|$ then there exists an
effective divisor $C$ such that $x\in C$ and
\begin{align}\label{Reider}
\overline{\mathcal{L}}_{n}\cdot C=0,\quad C^{2}=-1;\quad
\overline{\mathcal{L}}_{n}\cdot C=1,\quad C^{2}=0.
\end{align}
For $n$ big enough, the divisor $C$ in \ref{Reider} must satisfy
$\mathcal{L}\cdot C=0$. By the Hodge index theorem we then have
$C^{2}<0$. This rule out the second possibility in \ref{Reider}.
Regarding the remaining case, we argue as follows. Since
$\mathcal{L}\cdot C=0$, the divisor $C$ must be connected.
Furthermore, since $C^{2}=-1$ it is easy to see that such a divisor
must be reduced. But then we would have
\begin{align}\notag
K_{\overline{M}}\cdot C=0,\quad C^{2}=-1,
\end{align}
which contradicts the integrality of the \emph{arithmetic genus} of
$C$ \cite{Harthshorne}. Concluding, for $n$ big enough the linear
system $|n\mathcal{L}_{\alpha_{n}}|$ is base-point free. In other
words, the Kodaira map
\begin{align}\notag
i_{|K_{\overline{M}}+\overline{\mathcal{L}}_{n}|}:
\overline{M}\longrightarrow P^{N}
\end{align}
is everywhere defined for $n$ big enough. We then have that
$\mathcal{L}_{\alpha_{n}}$ can be represented by a smooth closed
Hermitian $(1, 1)$-form which is everywhere positive semi-definite.
Next, we want to understand the \emph{locus} where
$\mathcal{L}_{\alpha_{n}}$ fails to be a K\"ahler class. Again this
can be achieved by using Reider's theorem. More precisely, for $n$
big enough the only obstruction for the Kodaira map
$i_{|K_{\overline{M}}+\overline{\mathcal{L}}_{n}|}$ to be a local
diffeomorphism onto its image is given by the existence of effective
divisors $C$ such that
\begin{align}\notag
\mathcal{L}\cdot C=0,\quad K_{\overline{M}}\cdot C=0,\quad
C^{2}=\{-1, -2\}.
\end{align}
These divisors are now easily classified. In fact, by using the
reasoning given in the proof of Lemma \ref{lemma2}, we conclude that
that the only obstructions are given by the interior $(-2)$-curves
and the $(-2)$-curves in $D$ intersecting the other components of
$D$ in two points only.

\begin{proposition}\label{semi-positivity}
Let $(\overline{M}, D)$ be $D$-minimal of log-general type with $D$
a semi-stable curve. There exists $\overline{\alpha}\in (0, 1)$ such
that for any $\alpha\in[\overline{\alpha}, 1)$ then
$\mathcal{L}_{\alpha}$ can be represented by a smooth closed $(1,
1)$-form which is everywhere positive semi-definite and strictly
positive outside the interior $(-2)$-curves and the boundary
$(-2)$-curves intersecting the other components of $D$ in just two
points.
\end{proposition}
\begin{proof}
We have seen that, for $n$ big enough, there exists a representative
for $\mathcal{L}_{\alpha_{n}}$ which is everywhere positive
semi-definite and strictly positive outside the interior
$(-2)$-curves and the boundary $(-2)$-curves intersecting the other
components of $D$ in just two points. By Theorem 5.8. in
\cite{Sakai} we know that $\mathcal{L}$ is \emph{semi-ample}. We
then have that $\mathcal{L}$ can be represented by a positive
semi-definite smooth form. A simple computation now shows that, for
any $\alpha\in[\alpha_{n}, 1)$, there exist strictly positive real
numbers $\beta_{1}(\alpha)$ and $\beta_{2}(\alpha)$ such that
\begin{align}\notag
\mathcal{L}_{\alpha}=\beta_{1}\mathcal{L}_{\alpha_{n}}+\beta_{2}\mathcal{L}.
\end{align}
By letting $\overline{\alpha}$ be equal to $\alpha_{n}$, the proof
is then complete.
\end{proof}

\section{Applications}\label{applications}

In this section, we apply the recent analytical advances in the
theory of complex Monge-Amp\`ere equations with degenerate right
hand side \cite{Campana}, \cite{Mazzeo} to classify logarithmic
surfaces which admit K\"ahler-Einstein metrics with negative scalar
curvature and small edge singularities.

As in Section \ref{demailly}, let $D$ be a normal crossing divisor
and let $\{D_{i}\}$ be its irreducible components. For all $i$, let
denote by $L_{i}$ the line bundle associated to $D_{i}$ and let
$\sigma_{i}\in H^{0}(\overline{M},
\mathcal{O}_{\overline{M}}(L_{i}))$ be a defining section for
$D_{i}$. Finally, equip each of these line bundles with a Hermitian
metric $\{(L_{i}, \|\cdot\|)\}$.

Thus, if we are interested in constructing singular negative
K\"ahler-Einstein metrics on $\overline{M}\backslash D$ with
asymptotic behavior as in \ref{model}, given a K\"ahler class
$\omega$ on $\overline{M}$, it is natural to consider the following
singular complex Monge-Amp\`ere equation
\begin{align}\label{singular}
(\omega+\sqrt{-1}\partial\overline{\partial}\varphi)^{n}=e^{f+\varphi}\frac{\omega^{n}}{\prod_i\|\sigma_{i}\|^{2\alpha}}
\end{align}
whose right hand side is the volume form of an edge metric with cone
angle $2\pi(1-\alpha)$.

In fact, if we assume $\mathcal{L}_{\alpha}$ to be ample by choosing
$\omega\in[\mathcal{L}_{\alpha}]$ and $f$ such that
\begin{align}\label{lemmabar}
\sqrt{-1}(\partial\overline{\partial}\log{\omega^{n}}-\sum_{i}\alpha_{i}\partial\overline{\partial}\log{\|\sigma_{i}\|^{2}}
+\partial\overline{\partial}f)=\omega
\end{align}
if $\varphi$ is a solution of \ref{singular}, smooth outside $D$, it
is clear that $\omega_{\varphi}$ is a smooth K\"ahler-Einstein
metric with negative scalar curvature on $\overline{M}\backslash D$.

Equations of the type given in \ref{singular} where already studied
in the fundamental paper of S.-T. Yau \cite{Yau}. The approach
described in \cite{Yau} is through the study of non-singular
$\epsilon$-regularization of \ref{singular}. More precisely, one
tries to construct a solution of \ref{singular} by studying the
degeneration as $\epsilon\rightarrow 0$ of the solutions of
$\epsilon$-regularized equations of the form
\begin{align}\notag
(\omega+\sqrt{-1}\partial\overline{\partial}\varphi_{\epsilon})^{n}=
e^{f+\varphi_{\epsilon}}\frac{\omega^{n}}{\prod_i(\|\sigma_{i}\|^{2}+\epsilon^{2})^{\alpha}},
\end{align}
for more details see Section 8 in \cite{Yau}.

In \cite{Kolo}, \cite{Kolo1}, S. Ko\l odziej using techniques coming
from \emph{pluripotential} theory, developed a very general
existence, uniqueness and regularity theory for complex
Monge-Amp\`ere equations whose right hand side is a $L^{p}$-density
for some $p>1$. In particular, this very general theory can be
applied to solve equations like \ref{singular} for $\alpha\in(0,
1)$.

Finally, over the past year there has been a lot of progress towards
the completion of Tian's program \cite{Tian}. In fact R. Mazzeo, T.
Jeffres and Y. Rubinstein in \cite{Mazzeo}, building up on the work
of S. Donaldson \cite{Donaldson}, appear to have completed this
program in all dimensions, for all cone angles when $D$ is an
irreducible \emph{smooth} divisor. Mazzeo-Rubinstein announced the
resolution of this problem in the general case when $D$ has simple
normal crossings \cite{Mazzeo2}. These results make use of
Rubinstein's Ricci continuity method and Mazzeo's edge calculus
which in particular provides a fine asymptotic for the associated
K\"ahler-Einstein potentials, for more details see \cite{Mazzeo} and
the bibliography therein. Moreover, they have existence theorems
when $\mathcal{L}_{\alpha}\sim 0$ and $-\mathcal{L}_{\alpha}$ is
ample. For related results see also the works of R. Berman
\cite{Berman} and S. Brendle \cite{Brendle}. Interestingly, by using
an approach similar to the one suggested by S.-T. Yau in \cite{Yau},
F. Campana, H. Guenancia, M. P\u aun in \cite{Campana1} were able to
show the existence of a negatively curved K\"ahler-Einstein metric
with edges along a normal crossing divisor $D$ under the assumption
that $\mathcal{L}_{\alpha}$ is ample and $\alpha\in [\frac{1}{2},
1)$. We now use their main existence result, Theorem A in
\cite{Campana1} or alternatively Theorem 1.3. in \cite{Mazzeo2}, to
prove the following.

\begin{theorem}\label{Luca}
A logarithmic surface $(\overline{M}, D)$ admits negative
K\"ahler-Einstein metrics with edge singularities along $D$ and cone
angles $2\pi(1-\alpha)$ for all values of $\alpha$ close enough to
one iff $(\overline{M}, D)$ is $D$-minimal, log-general, $D$ is a
semi-stable curve and there are no interior $(-2)$-curves or
$(-2)$-curves in $D$ which intersects the other components of $D$ in
two points only.
\end{theorem}

\begin{proof}
By Proposition \ref{luca}, we know that if $(\overline{M}, D)$
admits a negative K\"ahler-Einstein metric with edge singularities
along $D$ with cone angle $2\pi(1-\alpha)$ then the associated
twisted log-canonical bundle $\mathcal{L}_{\alpha}$ has to be ample.
Thus, if $\mathcal{L}_{\alpha}$ is ample for all values of $\alpha$
close to one by Theorem \ref{luca2} we conclude that $(\overline{M},
D)$ is minimal, log-general, $D$ is a semi-stable curve and there
are no interior $(-2)$-curves or $(-2)$-curves in $D$ which
intersects the other components of $D$ in just two points.

Conversely, if $(\overline{M}, D)$ is as above by Theorem
\ref{luca2} we know that the $\mathcal{L}_{\alpha}$'s are ample for
all values of $\alpha$ close to one. Let us choose a K\"ahler class
$\omega$ in $[\mathcal{L}_{\alpha}]$ and $f$ as in \ref{lemmabar},
then by solving a singular complex Monge-Amp\`ere equation of the
form given in \ref{singular}, see for example \cite{Kolo},
\cite{Campana1}, we can construct a negative K\"ahler-Einstein
metric $\omega_{\varphi}$ on $\overline{M}\backslash D$. Now, since
we are working with values of $\alpha$ close to one we can assume
$\alpha\in[\frac{1}{2}, 1)$ and then applying Theorem A in
\cite{Campana1} or Theorem 1.3. in \cite{Mazzeo2} we conclude that
$\omega_{\varphi}$ is indeed \emph{quasi-isometric} to an edge
K\"ahler metric near $D$.
\end{proof}

We conclude this section by discussing the Bogomolov-Miyaoka-Yau
inequality for surfaces of log-general type. In \cite{TianY}, G.
Tian and S.-T. Yau were able to prove, among many other things, that
given a logarithmic surface $(\overline{M}, D)$ for which
$\mathcal{L}$ is big, nef and ample modulo $D$ then the inequality
\begin{align}\notag
c^{2}_{1}(\Omega^{1}_{\overline{M}}(\log{D}))\leq
3c_{2}(\Omega^{1}_{\overline{M}}(\log{D}))
\end{align}
holds. For the definition of the sheaf
$\Omega^{1}_{\overline{M}}(\log{D})$ and its basic properties we
refer to Chapter 3 in \cite{Griffiths}. As the reader can easily
verify, if $(\overline{M}, D)$ is log-general, $D$-minimal with $D$
a semi-stable curve and there are not interior $(-2)$-curves then
$\mathcal{L}$ is big, nef and ample modulo $D$. Following
\cite{Tian}, one may try to prove \ref{Yau} by deforming to zero the
cone angle of a family $\omega^{\alpha}_{\varphi}$ of negative
K\"ahler-Einstein metrics with edges singularities along $D$ and
apply a suitably modified Chern-Weil theory for this kind of
incomplete metrics. The theory developed by
Jeffres-Mazzeo-Rubinstein, see Theorem 2 in \cite{Mazzeo}, provides
a precise asymptotic for $\omega_{\varphi}$ near $D$. Remarkably,
this asymptotic behavior appears to be exactly what is needed in
order to develop a meaningful Chern-Weil theory in this context. In
fact, Atiyah and LeBrun in \cite{LeBrun} introduce the notion of
Riemannian \emph{edge-cone} metrics which are singular along
smoothly embedded codimension two submanifolds and derive the
analogues of the well-known \emph{Gauss-Bonnet} and \emph{signature}
formulas for closed $4$-manifolds. As the reader can easily verify,
if the cone angle is sufficiently small, the singular
K\"ahler-Einstein metrics constructed by Jeffres-Mazzeo-Rubinstein
have cone-edge singularities in the sense of Atiyah-LeBrun. Again
this follows from the deep analysis contained in Theorem 2. of
\cite{Mazzeo}. Thus, let us show that this approach does indeed work
when the boundary divisor $D$ is \emph{smooth}.

\begin{proposition}
Let $(\overline{M}, D)$ be $D$-minimal of log-general type without
interior $(-2)$-curves and let $D$ be a smooth semi-stable curve.
Then
\begin{align}\label{Yau}
c^{2}_{1}(\Omega^{1}_{\overline{M}}(\log{D}))\leq
3c_{2}(\Omega^{1}_{\overline{M}}(\log{D})).
\end{align}
\end{proposition}
\begin{proof}
By Theorem \ref{luca2}, if $(\overline{M},D)$ is as in the statement
then the twisted log-canonical bundle $\mathcal{L}_{\alpha}$ is
ample for all values of $\alpha$ close enough to one. By solving a
singular Monge-Amp\`ere equation as in \ref{singular}, see again
Theorem 2. in \cite{Mazzeo}, one can construct a family of singular
K\"ahler-Einstein metrics $\omega^{\alpha}_{\varphi}$ with negative
scalar curvature and cone angles $2\pi(1-\alpha)$ along $D$. For
$\alpha$ close enough to one, by Theorem 2.1. and Theorem 2.2. in
\cite{LeBrun} we know that
\begin{align}\notag
\chi(M,\omega^{\alpha}_{\varphi})=\chi(\overline{M})-\alpha\chi(D),
\quad
\sigma(M,\omega^{\alpha}_{\varphi})=\sigma(\overline{M})-\frac{1}{3}\alpha(2-\alpha)D^{2}
\end{align}
where $M=\overline{M}\backslash D$. For simplicity let us define
\begin{align}\notag
\chi_{\alpha}=\chi(M,\omega^{\alpha}_{\varphi}), \quad
\sigma_{\alpha}=\sigma(M,\omega^{\alpha}_{\varphi}).
\end{align}
Now, observe that
\begin{align}\notag
\mathcal{L}^{2}_{\alpha}=2\chi_{\alpha}+3\sigma_{\alpha}=
\frac{1}{4\pi^{2}}\int\Big(2|W_{+}|^{2}+\frac{s^{2}_{\omega^{\alpha}_{\varphi}}}{24}\Big)d\mu,
\end{align}
where $s$, $W_{+}$ and $W_{-}$ are the scalar curvature, the
self-dual and anti-self-dual Weyl curvatures of
$\omega^{\alpha}_{\varphi}$. Since $\omega^{\alpha}_{\varphi}$ is a
smooth K\"ahler metric on $\overline{M}\backslash D$, a \emph{local}
computation shows the pointwise equality
\begin{align}\notag
|W_{+}|^{2}=\frac{s^{2}}{24}
\end{align}
which therefore implies
\begin{align}\notag
\mathcal{L}^{2}_{\alpha}\leq3\Big(\frac{1}{4\pi^{2}}\int\frac{s^{2}}{24}d\mu\Big)
\leq\Big(\frac{1}{4\pi^{2}}\int2|W_{-}|^{2}+\frac{s^{2}}{24}d\mu\Big)=3(2\chi_{\alpha}-3\sigma_{\alpha}).
\end{align}
By letting $\alpha$ approach one, we conclude that
\begin{align}\notag
\chi(\overline{M})-\chi(D)\geq3(\sigma(\overline{M})-\frac{1}{3}D^{2}).
\end{align}
Moreover, by using the fact that $(\overline{M}, D)$ is a
logarithmic surface we have that
\begin{align}\notag
K^{2}_{\overline{M}}=2\chi(\overline{M})+3\sigma(\overline{M}),
\quad \chi(D)=-K_{\overline{M}}\cdot D-D^{2},
\end{align}
which implies
\begin{align}\label{Miyaoka}
3(\chi(\overline{M})-\chi(D))\geq \mathcal{L}^{2}.
\end{align}
The final step is to show the equivalence of \ref{Yau} and
\ref{Miyaoka}. First, let us consider the short exact sequence of
sheaves
\begin{align}\notag
0\longrightarrow\Omega^{1}_{\overline{M}}\longrightarrow\Omega^{1}_{\overline{M}}(\log{D}))
\longrightarrow\mathcal{O}_{D}\longrightarrow 0;
\end{align}
for more background see again Chapter 3 in \cite{Griffiths}. By the
Whitney product formula we have
\begin{align}\notag
c(\Omega^{1}_{\overline{M}}(\log{D}))=c(\Omega^{1}_{\overline{M}})\cdot
c(\mathcal{O}_{D}).
\end{align}
Since $D$ is reduced and effective, in order to compute
$c(\mathcal{O}_{D})$ we can simply apply the Whitney product formula
to the standard short exact sequence
\begin{align}\notag
0\longrightarrow\mathcal{O}_{\overline{M}}(-D)\longrightarrow\mathcal{O}_{\overline{M}}
\longrightarrow\mathcal{O}_{D}\longrightarrow 0.
\end{align}
In conclusion, we obtain
\begin{align}\notag
c(\Omega^{1}_{\overline{M}}(\log{D}))=(1+c_{1}(\Omega^{1}_{\overline{M}})+c_{2}(\Omega^{1}_{\overline{M}}))\cdot(1+D+D^{2})
\end{align}
which implies
\begin{align}\notag
c_{1}(\Omega^{1}_{\overline{M}}(\log{D}))=K_{\overline{M}}+D=\mathcal{L}
\end{align}
and with a slight abuse of notation
\begin{align}\notag
c_{2}(\Omega^{1}_{\overline{M}}(\log{D}))=c_{2}(\Omega^{1}_{\overline{M}})+K_{\overline{M}}\cdot
D+D^{2}=\chi(\overline{M})-\chi(D).
\end{align}
This concludes the proof of \ref{Yau} when the boundary divisor $D$
is smooth.
\end{proof}

It would be extremely interesting to extend this argument in the
case when $D$ is reduced with normal crossing. The first problem is
that, as shown in Theorem \ref{Luca}, not all pairs $(\overline{M},
D)$ for which $\mathcal{L}$ is big, nef and ample modulo $D$ admit a
$1$-parameter family of negative K\"ahler-Einstein metrics with
small edge singularities along $D$. Nevertheless, combining
Proposition \ref{semi-positivity} with Theorem 6.1. in \cite{Pali},
one can still construct a $1$-parameter family of negative
K\"ahler-Einstein metrics $\gamma^{\alpha}_{\varphi}$ on
$\overline{M}\backslash D$. More precisely, given a logarithmic
surface  $(\overline{M},D)$ as in Proposition \ref{semi-positivity}
and without interior $(-2)$-curves, for any
$\alpha\in[\overline{\alpha},1)$ let $\gamma_{\alpha}$ be a smooth
semi-positive representant for the $\setR$-cohomology class
$[\mathcal{L}_{\alpha}]$. This smooth form is \emph{strictly}
positive outside the boundary $(-2)$-curves intersecting the other
components of $D$ in two points only. Now, let $\Omega$ be a smooth
volume form on $\overline{M}$ and consider the family of degenerate
complex Monge-Amp\`ere equations
\begin{align}\label{degenerate}
(\gamma_{\alpha}+\sqrt{-1}\partial\overline{\partial}\varphi)^{n}=
e^{f+\varphi}\frac{\Omega}{\prod_i\|\sigma_{i}\|^{2\alpha}}
\end{align}
for any $\alpha\in[\overline{\alpha},1)$. By Theorem 6.1. in
\cite{Pali}, for a fixed $\alpha$ this degenerate equation admits a
unique solution $\varphi\in L^{\infty}(\overline{M})$ which is
smooth on $\overline{M}\backslash D$. Moreover, by appropriately
choosing the function $f$ in \ref{degenerate} we can arrange
$\gamma^{\alpha}_{\varphi}$ to be Einstein with negative scalar
curvature on $\overline{M}\backslash D$. Thus, the remaining problem
is to give a \emph{topological} interpretation of the curvature
integrals
\begin{align}\notag
\chi_{\alpha}=\chi(M,\gamma^{\alpha}_{\varphi}), \quad
\sigma_{\alpha}=\sigma(M,\gamma^{\alpha}_{\varphi}).
\end{align}
Conjecturally, the same elegant formulas given in Theorem 2.1. and
Theorem 2.2. of \cite{LeBrun} hold. The proof of \ref{Yau} should
then follow as in the smooth boundary divisor case.

Concluding, a Chern-Weil approach to \ref{Yau}, through deformations
of negative K\"ahler-Einstein metrics with singularities along $D$,
has to rely on a generalization of the recent theory of M.F. Atiyah
and C.LeBrun \cite{LeBrun} outside the realm of Riemannian metrics with pure cone-edge asymptotic.\\

\noindent\textbf{Acknowledgements}. I would like to thank Samuel
Grushevsky, Thalia Jeffres, Blaine Lawson, Claude LeBrun and Mark
Stern for useful discussions.



\end{document}